\documentclass[10pt, a4paper]{article}
\usepackage{graphicx,latexsym, amsmath,amsfonts}
\usepackage{amsthm}
\usepackage{enumerate}

\usepackage{epsfig}
\usepackage{color}

\begin{document}


\renewcommand{\d}{\mathrm{d}}

\newcommand{\E}{\mathbb{E}}
\newcommand{\PP}{\mathbb{P}}
\newcommand{\RR}{\mathbb{R}}
\newcommand{\NN}{\mathbb{N}}
\newcommand{\fracs}[2]{{ \textstyle \frac{#1}{#2} }}

\newtheorem{theorem}{Theorem}[section]
\newtheorem{lemma}[theorem]{Lemma}
\newtheorem{coro}[theorem]{Corollary}
\newtheorem{defn}[theorem]{Definition}
\newtheorem{assp}[theorem]{Assumption}
\newtheorem{expl}[theorem]{Example}
\newtheorem{prop}[theorem]{Proposition}
\newtheorem{rmk}[theorem]{Remark}
\newtheorem{notation}[theorem]{Notation}

\def\a{\alpha} \def\g{\gamma}
\def\e{\varepsilon} \def\z{\zeta} \def\y{\eta} \def\o{\theta}
\def\vo{\vartheta} \def\k{\kappa} \def\l{\lambda} \def\m{\mu} \def\n{\nu}
\def\x{\xi}  \def\r{\rho} \def\s{\sigma}
\def\p{\phi} \def\f{\varphi}   \def\w{\omega}
\def\q{\surd} \def\i{\bot} \def\h{\forall} \def\j{\emptyset}

\def\be{\beta} \def\de{\delta} \def\up{\upsilon} \def\eq{\equiv}
\def\ve{\vee} \def\we{\wedge}

\def\t{\tau}

\def\F{{\cal F}}
\def\T{\tau} \def\G{\Gamma}  \def\D{\Delta} \def\O{\Theta} \def\L{\Lambda}
\def\X{\Xi} \def\S{\Sigma} \def\W{\Omega}
\def\M{\partial} \def\N{\nabla} \def\Ex{\exists} \def\K{\times}
\def\V{\bigvee} \def\U{\bigwedge}

\def\1{\oslash} \def\2{\oplus} \def\3{\otimes} \def\4{\ominus}
\def\5{\circ} \def\6{\odot} \def\7{\backslash} \def\8{\infty}
\def\9{\bigcap} \def\0{\bigcup} \def\+{\pm} \def\-{\mp}
\def\<{\langle} \def\>{\rangle}

\def\lev{\left\vert} \def\rev{\right\vert}
\def\1{\mathbf{1}}

\newcommand\wD{\widehat{\D}}

\newcommand{\ls}[1]{\textcolor{red}{\tt Lukas: #1 }}

\title{ \bf Discretizing the Heston Model: \\ An Analysis of the Weak Convergence Rate}

\author{Martin Altmayer \footnote{TNG Technology Consulting GmbH,
Betastra{\ss}e 13a,
D-85774 Unterf\"ohring, Germany}
 \and  Andreas Neuenkirch  \footnote{Institut f\"ur Mathematik, Universit\"at Mannheim,  A5, 6, D-68131 Mannheim, Germany, \texttt{neuenkirch@kiwi.math.uni-mannheim.de  }} 
}

\date{}

\maketitle

\begin{abstract}
In this manuscript we analyze the weak convergence rate of a discretization scheme for the Heston model. Under mild assumptions on the smoothness of the payoff and on 
the Feller index of the volatility process, respectively,  we establish  a weak convergence  rate of  order one.  Moreover, under almost minimal assumptions we obtain weak convergence without a rate. These results are accompanied by several numerical examples.
Our error analysis relies on
a classical  technique from Talay and Tubaro \cite{TT}, 
a recent regularity estimate for the Heston PDE  \cite{FP} and Malliavin calculus. 

\medskip
\noindent \textsf{{\bf Key words: } \em Heston model, discretization schemes for SDEs, Kolmogorov PDE, Malliavin calculus \\
}

\end{abstract}

\section{Introduction and Main Results}
\label{sec;introduction}

The Heston model \cite{heston} is given by the stochastic differential equation (SDE)
\begin{align}  \begin{split}
\d S_t &= \mu S_t \d t + \sqrt{v_t} S_t  (\rho \d W_t + \sqrt{1-\rho^2} \d B_t), \qquad    \, \,  t \in [0,T], \\
\d V_t &= \kappa(\lambda - V_t) \d t + \theta \sqrt{V_t} \d W_t,  \qquad  \qquad \qquad \qquad t \in [0,T], \label{hes-eq}
\end{split}
\end{align}
with $S_0, V_0, \kappa, \lambda, \theta >0$, $\mu \in \mathbb{R}$,  $\rho \in [-1,1]$ and independent Brownian motions $W,B$. It is a simple and popular extension of the Black--Scholes model.
Here $S$ models the price of an asset and $V$ its volatility, which is given by the so called  Cox--Ingersoll--Ross process (CIR). 

\medskip

While numerous discretization schemes and simulation methods for SDE \eqref{hes-eq} have been proposed and numerically tested, see e.g.  \cite{KJ,And-H,NN,lord-comparison,AA2,GlaKim}, an analysis of the weak convergence rate has not been carried out so far\,---\,up to the best of our knowledge.
In this manuscript we are addressing this gap by analyzing a numerical scheme, which uses the drift-implicit Milstein scheme  \cite{KGJ} for the volatility and an Euler discretization for the log-Heston
price. Our approach relies on a recent regularity result for the Heston PDE \cite{FP},  tail estimates for the CIR process, the Kolmogorov PDE approach for the weak error analysis from \cite{TT} and Malliavin calculus tools. 
It is crucial that the scheme is built on a positivity preserving discretization of the CIR process,
\begin{itemize}
 \item[(i)] since the domain of the Kolmogorov PDE is restricted to non-negative values
of the volatility, 
\item[(ii)] since the positivity of the discretization scheme allows to establish required estimates of its inverse moments.
\end{itemize}
Note that SDE \eqref{hes-eq} can be simulated exactly, an algorithm for this was given by Broadie and Kaya in \cite{BK}.
Nevertheless discretization schemes for the Heston model are important and interesting for at least two reasons: (i) 
they can be easily extended to multidimensional versions of the Heston model consisting of $d$ assets (for which exact simulation methods are unknown), and  (ii) the method given in \cite{BK} still requires the numerical inversion of a characteristic function, which turns out to be a computational bottleneck.

\medskip

It is common numerical practice to consider the log-Heston model  instead of the Heston model. The transformation $X_t=\log(S_t)$  yields the SDE
\begin{align} \begin{split}
 \d X_t &= \big(\mu-\frac12 V_t\big ) \d t + \sqrt{V_t} \d (\rho W_t + \sqrt{1-\rho^2} B_t), \\
 \d V_t &= \kappa (\lambda-V_t) \d t + \theta \sqrt{V_t} \d W_t, \label{hes-log} \end{split}
\end{align}
with $X_0=x_0=\log(S_0) \in \mathbb{R}$, $V_0=v_0 >0$,
and the exponential is then incorporated in the payoff $g:[0,\infty) \rightarrow \mathbb{R}$, i.e.\ $g$ is replaced by
$f: \mathbb{R} \rightarrow  \mathbb{R}$ with $f(x)=g(\exp(x))$.

\medskip

To analyse the  convergence rate, we will work under the following assumption on the payoffs and the parameters of the CIR process (for a discussion see Remarks \ref{rem-S} and \ref{rem-Feller}):
{\it \begin{itemize}
 \item[\textrm{(S)}] The function $f\colon \mathbb{R} \rightarrow \mathbb{R}$ is twice continuous differentiable with compact support. Moreover, there exists an $\varepsilon >0$ such that $f''\colon \mathbb{R} \rightarrow \mathbb{R}$ is H\"older continuous of order 
 $\varepsilon$, i.e.\ $f''$ satisfies
 $$ \sup_{x,y \in \mathbb{R}, \, x \neq y} \frac{|f''(x)-f''(y)|}{|x-y|^{\varepsilon}} < \infty $$
 \item[\textrm{(F)}] We have $$ \nu:= \frac{2\kappa \lambda}{\theta^2}>2$$
\end{itemize}} 

\medskip
 
The  scheme  we consider consists of a drift-implicit Milstein scheme for the volatility and  an Euler scheme for the log-price:
\begin{align*}
{} \qquad {x}_0 &= x_0, \qquad {v}_0 = v_0, \\
(D) \qquad {x}_{n+1} &= {x}_n + \Big(\mu-\frac12 {v}_n\Big) (t_{n+1}-t_n)        + \sqrt{{v}_n} \Big( \rho \Delta_n W  + \sqrt{1-\rho^2} \Delta_n B \Big), \\
{} \qquad {v}_{n+1} &= {v}_n + \kappa(\lambda-{v}_{n+1}) (t_{n+1}-t_n) + \theta \sqrt{{v}_n} \Delta_n W + \frac{\theta^2}{4} \big((\Delta_n W)^2 - (t_{n+1}-t_n)   \big)
\end{align*}
Here $$0 = t_0 < t_1 < \dots < t_N  = T$$ is a discretization of $[0,T]$ and we use the abbreviations $$\Delta_n B = B_{t_{n+1}} - B_{t_n} \qquad  \textrm{and}  \qquad \Delta_n W = W_{t_{n+1}} - W_{t_n}.$$
This scheme is well defined, iff $4 \frac{ \kappa \lambda}{\theta^2} \geq 1$, since
the discretization of the CIR process can be written as
$$ {v}_{n+1} =  \frac{1}{1+ \kappa (t_{n+1}-t_n)} \left( \Big( \sqrt{{v}_{n}} + \frac{\theta}{2} \Delta W_n \Big)^2 +  \Big( \kappa \lambda - \frac{\theta^2}{4}\Big)(t_{n+1}-t_n) \right),$$
and thus $v_n \geq 0$, $n=0,1, \ldots$.
\medskip


In the following we use the notations $$ \Delta=\max_{k=1, \ldots, N} |t_{k}-t_{k-1}|$$
for the maximal stepsize and
$$ e(f;\Delta)= | Ef({x}_N) - Ef(X_T) | $$
for the weak error.
\medskip

\begin{theorem}\label{theorem-main} Assume  (S)  and  (F). Then, for all $\alpha \in (0,1)$ the scheme
(D)  satisfies
$$   \lim_{\Delta \rightarrow 0} \frac{e(f;\Delta)}{\Delta^{\alpha}} =0  $$   
\end{theorem}

\medskip

For the weak convergence result without a rate we will assume on the Feller index that:
{\it 
\begin{itemize}
 \item[(F-min)] We have $$ \nu:= \frac{2\kappa \lambda}{\theta^2}> \frac{1}{2}$$
 \end{itemize}
}
 \medskip

\begin{theorem}
 \label{theorem-weak} Assume    (F-min) and let $f \in C(\mathbb{R} \setminus O ; \mathbb{R})$ with $O \subset \mathbb{R}$ a finite set. Moreover assume  that
 $$ \textrm{(Int)} \qquad  \limsup_{\Delta \rightarrow 0} E |f(x_N)|^{1+ \varepsilon} < \infty $$
 for some $\varepsilon >0$. Then (D) satisfies
 $$ \lim_{\Delta \rightarrow 0} e(f;\Delta) = 0$$
\end{theorem}

\medskip If the correlation $\rho$ is negative, i.e.\ $\rho <0$,
assumption  {\it (Int) } is  satisfied e.g.\ for European call options, i.e.\ $f(\cdot )=(\exp(\cdot)-K)^+$,  and more generally  for $g \leq \textrm{const} \cdot \textrm{id}$. A negative correlation often appears in practice, see e.g. \cite{kimmel,BK}.
\medskip

\begin{prop}
 \label{prop-int} Assume    (F-min) and let $f\colon \mathbb{R} \rightarrow \mathbb{R}$ be such that $$ \sup_{x \in \mathbb{R}} |f(x)\exp(-x)| < \infty$$ If $\rho < 0$, then (Int) is satisfied.
\end{prop}


\subsection{Remarks}

\smallskip
%
%
%

\smallskip

\begin{rmk} 
Theorem  \ref{theorem-main} states that the weak error converges faster than any order $\alpha <1$. For payoffs with compact support in the log-asset price and the volatility we obtain in estimate
\eqref{order_delta} weak convergence order $\alpha=1$. The slightly weaker statement in Theorem  \ref{theorem-main}  is due to the additional use of  tail estimates for the CIR process to avoid the  compact support assumption for the volatility.
\end{rmk}

\begin{rmk}
The weak approximation of the CIR process has been analyzed by Alfonsi in \cite{AA1} and \cite{AA2}. In \cite{AA1}
he shows\,---\,among other results\,---\,that several schemes have weak order one if $f \in C^{4}_{pol}(\mathbb{R}; \mathbb{R}) $ and $\nu \geq 1/2$, respectively $\nu  \geq 1$, depending on the considered scheme.
In \cite{AA2} he constructs second and third order schemes for the CIR process for $f \in C^{\infty}_{pol}(\mathbb{R}; \mathbb{R})$ and without a restriction on the Feller index.

The notation $C^{k}_{pol}(\mathbb{R}; \mathbb{R})$ stands here for the subset of functions of  $C^k(\mathbb{R}; \mathbb{R})$, which have polynomially bounded derivatives up to order $k \in \mathbb{N}$.
\end{rmk}

\smallskip

\begin{rmk}\label{rem-S} Payoffs  in mathematical  finance are typically at most Lipschitz continuous, thus the smoothness conditions of (S) are in general not satisfied. Assumption (S)
arises from using the results from \cite{FP}, see Section 3, which give estimates for the smoothness of the Kolmogorov PDE. In \cite{MA} a weak error analysis for the scheme {\it (D)} has been given by the first author for payoffs  which are 
only bounded and measurable. Weak order one is established there, however the analysis requires the restriction $\n> \frac{9}{2}$ on the Feller index $\nu = \frac{2\kappa \lambda}{\theta^2} $.

A boundedness assumption (which is implied by (S)) for the payoff or assumption (Int) is typical for a convergence rate analysis, since
the Heston model admits moment explosions, i.e.\ $E(S_T^p)=\infty$ for certain parameter constellations and $p>1$, see e.g.\ \cite{And-M}.
\end{rmk}

\smallskip

\begin{rmk}\label{rem-BT}  In a seminal work   Bally and Talay (\cite{Talay}) analyse the weak error of the Euler scheme for test functions (i.e.\ payoffs in our setting), which  are only bounded and measurable.
Using Malliavin calculus techniques they establish a weak error of order one (together with an error expansion) for such test functions, if the considered SDE
has smooth coefficients and
 additionally satisfies a non-degeneracy condition
of H\"ormander type. The latter assumptions are not met for the Heston model.

Kebaier \cite{Keb}  illustrates the necessity of the non-degeneracy condition. He constructs an SDE with smooth coefficients but degenerated support of the law and $C^1$-test functions $f_{\alpha}$ such that the weak error of the Euler scheme
is of exact order $\alpha \in [1/2,1)$.

\end{rmk}
\smallskip

\begin{rmk}\label{rem-Feller} The assumption $\nu >2$ on the Feller index ensures that the inverse of our volatility approximation ${v}_n$
has a finite first moment, which is needed in our error analysis. Note that the inverse of $V_t$, i.e.\ of the CIR process itself, has a finite first moment iff $\nu >1$.

The Feller index controls the probability distribution of $V_t$. The smaller it is, the more likely $V_t$ takes values close to zero. The results given in \cite{AA1,AA2,MA} and here indicate that there
is a tradeoff in the error analysis between the smoothness assumptions on $f$ and  the restriction on the Feller index: the more smoothness on  $f$ is assumed, the smaller is the restriction on $\nu$.
\end{rmk}

\smallskip

\section{Numerical Results}
In this section we will present numerical results which indicate that for the scheme {\it (D)} a weak error rate of order one is typically reached even under milder assumptions than {\it (S)} and {\it (F)} -- as so often when 
a weak and strong error analysis of the CIR process respectively the Heston model is carried out, see e.g. \cite{AA1,AA2,AN,MA}.

We use model parameters from \cite{kimmel} (Model 1) and \cite{BK} (Model 2 and 3):
\bigskip\\
Model 1: $T=2$, $\mu=0$, $\kappa=5.07$, $\lambda=0.0457$, $\theta=0.48$, $\rho=-0.767$, $S_0=100$, $V_0=\lambda$.\smallskip\\
Model 2: $T=1$, $\mu=0.0319$, $\kappa = 6.21$, $\lambda = 0.019$, $\theta = 0.61$, $\rho = -0.7$, $S_0 = 100$, $V_0=0.010201$.\smallskip\\
Model 3: $T=5$, $\mu = 0.05$, $\kappa = 2$, $\lambda = 0.09$, $\theta = 1$, $\rho = -0.3$, $S_0 = 100$, $V_0 = 0.09$.\\

Note that the Feller index is $\nu = 2\kappa\lambda/\theta^2 \approx 2.01$ in the first model, $\nu\approx 0.63$ in Model 2 and $\nu \approx 0.34$ in the third model. In the letter case, our approximations of the CIR process might become negative. Here we replace $\sqrt{v_n}$ by $\sqrt{v_n^+}$ in {\it (D)}.

We use the following functionals, all depending on a parameter $K \in \mathbb{R}$. 
\begin{enumerate}
 \item Put: $f_1(x) = e^{-\mu T}(K-x)^+$.
 \item Smoothed put: $f_2(x) = f_1(x)$ for $x \not\in [0.9\cdot K,1.1 \cdot K]$. Inside the interval $[0.9\cdot K,1.1 \cdot K]$ the function $f_2$ is given by a polynomial whose function values and first, second, and third order derivatives coincide with those of $f_1$ at $0.9\cdot K$ and $1.1\cdot K$.
 \item Indicator: $f_3(x) = e^{-\mu T} 1_{[0,K]}(x)$.
\end{enumerate}

To maximize the influence of the irregularity of the functional we set $K=S_0$.
In order to measure the weak error rate, we have simulated at least $2\cdot 10^7$ samples of $f(S_T^\Delta)$ for each combination of model parameters, functional and number of steps $N \in \{2^0,\dots,2^{8}\}$, where $\Delta=T/N$. The mean of these samples was then compared to a reference solution and the resulting error (depending on $\Delta=T/N$) is plotted in Figures \ref{fig:plot1}-\ref{fig:plot3}. For the put and indicator functionals semi-exact formulae are available and have been used to compute the reference solution. In fact, the put price can be computed from the call price formula given in \cite{heston} and the well-known put-call parity. The price of the digital option can be computed from the probability $P_2$ given in \cite{heston}; it equals $e^{-\mu T}\cdot(1-P_2)$. For the smoothed put such a formula is not available and the reference solution was computed using (at least) $2\cdot 10^7$ samples with $2^{10}$ steps. Each curve is accompanied by a least-squares fit whose slope was used to measure the rate of convergence. The results can be found in Table \ref{tbl:convergence_rates}.

\begin{table}[h] \centering 

\begin{tabular}{l | c | c | c | c}
        & $\nu$    & Smoothed Put & Put & Indicator\\ \hline \hline 
Model 1 & $2.01$ & $0.62$ & $0.58$ & $1.01$ \\ \hline 
Model 2 & $0.63$ & $1.00$ & $0.91$ & $1.02$ \\ \hline 
Model 3 & $0.36$ & $0.96$ & $0.90$ & $0.88$ \\
\end{tabular}

\caption{Measured convergence rates.} \label{tbl:convergence_rates}
\end{table}

\begin{figure}[h] \centering
 \includegraphics[width=0.70\linewidth]{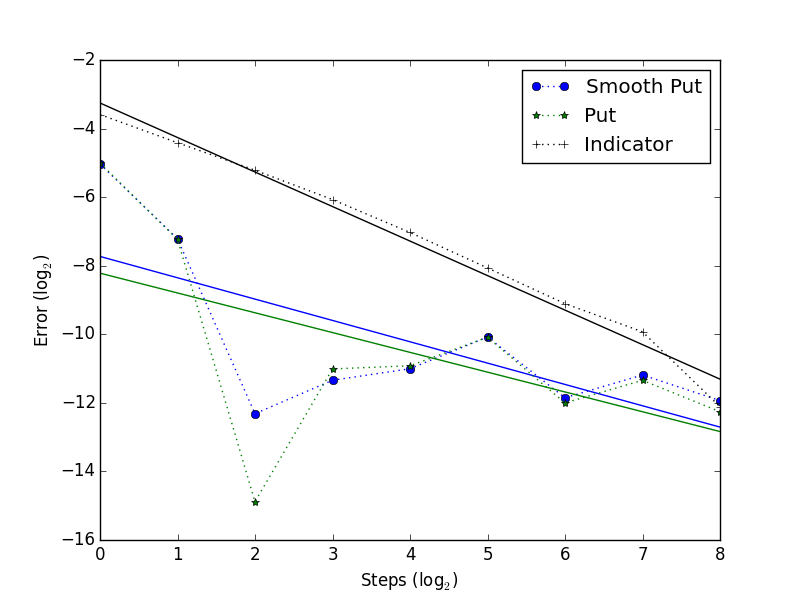}
 \caption{Weak error in Model 1.}   \label{fig:plot1} 
\end{figure}

It turns out that the most regular behavior is obtained in Model 2: For all three functionals the error decays with order one. Because the Feller index is only about $0.63$, this indicates that the assertion of Theorem \ref{theorem-main} also holds under weaker assumptions. In Model 3, which has an even lower Feller index, the error decay is weaker and less regular. Also, the rate now decreases slightly when the functional becomes less smooth.

\begin{figure}[h] \centering
 \includegraphics[width=0.70\linewidth]{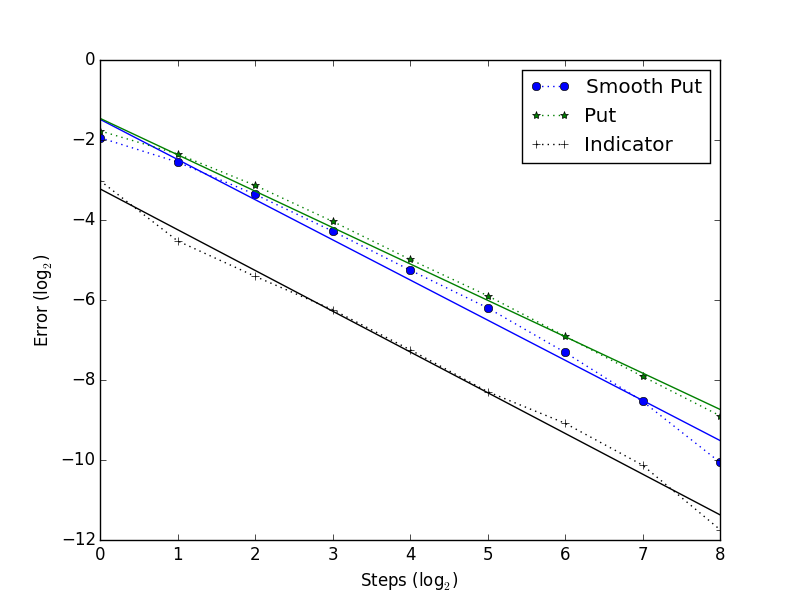}
 \caption{Weak error in Model 2.}
\end{figure}

Model 1 has the highest Feller index $\nu\approx 2.01$, thus satisfies {\it (F)}, and is the only model to fulfill the differentiability assumptions of Theorem \ref{theorem-main}. Surprisingly though, the error of the put functionals decays very irregular in this model and weak order one can only be observed for the indicator functional.
On first thought, this behaviour seems to violate Theorem \ref{theorem-main}. However, a closer look at the error of the put functionals, in particular for $N\in\{2^2,2^3,2^4\}$, reveals that this error is much smaller in Model 1 (approx.\ $2^{-12}$) than in Models 2 and 3 (within $[2^{-8},2^{-2}]$).
A comparison with the indicator functional in Model 1 shows that the reason for the low measured rate is simply the fact that in Model 1 a small number of steps is already sufficient to approximate the put functionals with an astonishingly high precision.

\begin{figure}[h]  \centering 
 \includegraphics[width=0.70\linewidth]{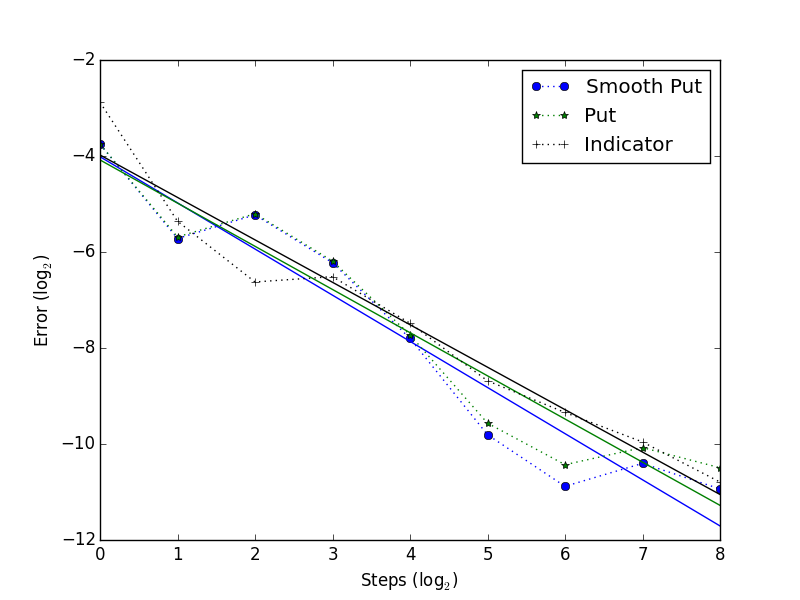}
 \caption{Weak error in Model 3.} \label{fig:plot3}
\end{figure}

\section{Auxiliary Results}
\label{sec;preliminary}

In this section we will collect and establish, respectively, several auxiliary results for the weak error analysis.
Without loss of generality  we can assume in the following $\mu =0$ 
by replacing $f$ with $f(\cdot + \mu T)$. 

\subsection{Kolmogorov PDE}
 In our error analysis we will follow the now classical approach of \cite{TT}, which exploits
the regularity of the Kolmogorov backward equation for
$$u(t,x,v) := E(h(X_{T}^{t,x,v},V_T^{t,v})), \qquad t \in [0,T], \, x \in \mathbb{R}, \, v \geq 0$$
Here
\begin{align*}
 X_s^{t,x,v} &= x -\frac12 \int_t^s V_r^{t,v} \d r +  \int^s_t \sqrt{V_r^{t,v}} \d \big(\rho W_r + \sqrt{1-\rho^2} B_r\big), \quad & s \geq t\\
 V_s^{t,v} &= v+ \int_t^s \kappa (\lambda-V_r^{t,v}) \d r + \theta \int^s_t \sqrt{V_r^{t,v}} \d W_r, \quad & s \geq t
\end{align*}  
and by an application of the Feynman--Kac theorem (see e.g. Theorem 5.7.6 in \cite{KS}) we obtain for $h\colon \mathbb{R} \times [0, \infty) \rightarrow \mathbb{R}$ bounded and continuous that $u$ satisfies
\begin{align} \label{pde1}
\partial_t u(t,x,v)  = & \frac v 2  \partial_x u(t,x,v)  - \kappa(\lambda-v)  \partial_v u(t,x,v)  \\ & \,\, - \frac v 2 \left(  \partial_{xx} u(t,x,v) + 2 \rho\theta  \partial_{xv} u (t,x,v)  + \theta^2 \partial_{vv} u (t,x,v) \right)  \nonumber, \qquad
t \in (0,T), \, x \in \mathbb{R},\,  v > 0 \end{align} with terminal condition
\begin{align} u(T,x,v)=h(x,v), \qquad x \in \mathbb{R}, \, v \geq 0\label{pde2} \end{align}
Due to the presence of the variable $v$ in front of the second order partial derivatives this partial differential equation (PDE) is a degenerate parabolic  equation for which a-priori regularity estimates on $[0,T]\times \mathbb{R} \times [0,\infty)$ have
been only recently established in \cite{FP}. To deal with the degeneracy of the differential operator Feehan and Pop use the cyclodical distance $d_c$ (see e.g.\ \cite{DH}) given by
\begin{align*}
& d_c((t_1,x_1,v_1),(t_2,x_2,v_2)) \\ & \qquad \qquad   \qquad := \frac{|x_1-x_2|+|v_1-v_2|}{\sqrt{v_1}+\sqrt{v_2}+\sqrt{|x_1-x_2|}}+\sqrt{|t_1-t_2|}, \qquad (t_i,x_i,v_i) \in  \mathcal{D}, \, i=1,2 
\nonumber \end{align*}
with $\mathcal{D} \subseteq [0,T]\times \mathbb{R} \times [0, \infty)$
and the Euclidean distance 
\begin{align*}
& d_e((t_1,x_1,v_1),(t_2,x_2,v_2))  \\   & \qquad \qquad   \qquad  := |x_1-x_2|+|v_1-v_2|+ \sqrt{|t_1-t_2|},  \qquad (t_i,x_i,v_i) \in  {\mathcal{D}}, \, i=1,2 \nonumber
\end{align*}
Furthermore set $\mathcal{D}_1= [0,T]\times \mathbb{R} \times [0, 1]$ and   $\mathcal{D}_2= [0,T]\times \mathbb{R} \times [1, \infty)$.
Roughly spoken the main result (Theorem 1.1) of \cite{FP} states that, if the terminal condition is smooth enough, i.e. twice continuously differentiable with   $\varepsilon$-H\"older continuous  second order derivatives, and has compact support, then  the solution $u$ to the Kolmogorov backward PDE has the following properties:
\begin{itemize}
 \item[(i)]  On $\mathcal{D}_2$, i.e.~if $v$ is bounded away from zero, then $u,\partial_t u,\partial_v u,\partial_x u,\partial_{xx} u,\partial_{xv} u$ and $\partial_{vv} u$ are bounded and H\"older continuous of  order $\varepsilon$ with respect to $d_{e}$.
 \item[(ii)] On $\mathcal{D}_1$, i.e.~for $v$ close to zero, then $u,\partial_t u,\partial_v u,\partial_x u$ and the damped second order derivatives  $v \partial_{xx} u,v \partial_{xv} u$ and $v \partial_{vv} u$ are bounded and H\"older continuous of  order $\varepsilon$ with respect to $d_{c}$.
\end{itemize}
For us, it will be sufficient to use the following result, which states a control for the (damped) derivatives of  $u$ and which is a direct consequence of Theorem 1.1 in \cite{FP}.
To state the result, let $M>0$ be sufficiently large and  let $\phi_M \in C^{3}([0, \infty); [0, \infty))$ be  functions such that 
\begin{itemize}
 \item[(i)] $\sup_{v \in (0, \infty)}   \left| \left(\frac{d}{dv}\right)^k\phi_M(v) \right|  \leq 1  $ for $ k \in \{0,1,2,3\}$
 \item[(ii)] $\phi_M(v)=1$ for  $v \leq M$ 
 \item[(iii)] $\phi_M(v)=0$ for $ v \geq 2M$
\end{itemize}

\medskip

\begin{theorem} \label{FP-main} {\it 
 Let $f: \mathbb{R} \rightarrow  \mathbb{R}$  satisfy (S) and let $M>0$ be such that $\{ x \in \mathbb{R}: \, f(x) \neq 0 \} \subset [-2M,2M]$.
 Then, there exist $q >0$ and  $c(f,\varepsilon,q)>0$, which are in particular independent of $M$, such that the solution $u$ to \eqref{pde1} and \eqref{pde2} with
 right hand side $h(x,v)=f(x) \phi_M(v), x \in \mathbb{R}, v \geq 0$ satisfies
\begin{align*} 
 &\sup_{(t,x,v) \in \mathcal{D}_1 \cup \mathcal{D}_2} \big(|u(t,x,v)| + |\partial_t u(t,x,v)| + |\partial_v u(t,x,v)|+   |\partial_x u(t,x,v)| \big) \leq  c(f,\varepsilon,q)(1+M^q) \\
 &\sup_{(t,x,v) \in  \mathcal{D}_1} \big(|v\partial_{xx} u(t,x,v)| + |v\partial_{xv} u(t,x,v)| + |v\partial_{vv} u(t,x,v)|\big) \leq   c(f,\varepsilon,q)(1+M^q) \\
 &\sup_{(t,x,v) \in \mathcal{D}_2} \big(|\partial_{xx} u(t,x,v)| + |\partial_{xv} u(t,x,v)| + |\partial_{vv} u(t,x,v)| \big)  \leq   c(f,\varepsilon,q)(1+M^q)
\end{align*}
}
\end{theorem}

\bigskip

\subsection{Malliavin calculus}
To establish our main results, we will use a Malliavin integration by parts procedure, see Lemma \ref{lem:ibp}. Otherwise, we would require stronger smoothness assumptions on the payoffs to obtain a weak convergence order of one,
or would obtain a non-sharp convergence rate. This paragraph gives a short introduction into Malliavin calculus, for more details we refer to \cite{nualart}.

\medskip

Malliavin calculus adds a derivative operator to stochastic analysis. Basically, if $Y$ is a random variable and $(W_t,B_t)_{t\in[0,T]}$ a two-dimensional Brownian motion, then the Malliavin derivative  measures
the dependence of $Y$ on $(W,B)$. The Malliavin derivative is defined by a standard extension procedure:
Let $\mathcal{S}$ be the set of smooth random variables of the form
$$ S= \varphi \left( \int_0^T h_1(s) \d (W_s,B_s), \ldots ,\int_0^T h_k(s) \d (W_s,B_s) \right)$$ with  $ \varphi \in C^{\infty}(\mathbb{R}^{k};\mathbb{R})$ bounded with bounded derivatives,
$h_i \in L^2([0,T]; \mathbb{R}^2)$, $i=1, \ldots, k$, and the stochastic integrals
$$ \int_0^T h_j(s) \d (W_s,B_s)= \int_0^T h_j^{(1)}(s) \d W_s + \int_0^T h_j^{(2)}(s) \d B_s$$
The derivative operator $D$ of such a smooth  random variable is defined as 
$$ D S = \sum_{i=1}^{k} 
\frac{\partial \varphi}{\partial x_{i}} \left( \int_0^T h_1(s) \d (W_s,B_s), \ldots, \int_0^T h_k(s) \d (W_s,B_s) \right)  h_i $$
This operator is closable from $L^{p}(\Omega)$ into $L^{p}\big(\Omega; L^2([0,T]; \mathbb{R}^2) \big)$ and the Sobolev space $\mathbb{D}^{1,p}$ denotes
the closure of  $\mathcal{S}$ with respect to the norm
$$\| Y \|_{1,p} \; = \;  \left(  E |Y|^p  +  {E} \left| \int_0^T |D_s Y|^2 \d s \right|^p \right)^{1/p}$$
In particular, if $D^{W}$ denotes the first component of the Malliavin derivative, i.e.\ the derivative with respect to $W$,
we have  $$D^{W}_t Y = \left \{ \begin{array}{clc} 1_{[0,t]} & \textrm { if } & Y=W \\ 0 & \textrm{ if } &  Y=B\end{array} \right. $$ and vice versa for the derivative with respect to $B$, i.e.
 $$D^{B}_t Y = \left \{ \begin{array}{clc} 1_{[0,t]} & \textrm { if } & Y=B  \\ 0 & \textrm{ if } &  Y=W\end{array} \right. $$  This in particular implies that if $Y \in  \mathbb{D}^{1,2}$ is independent of $W$,
 then $D^{W}Y=0$.

The derivative operator follows rules similar to ordinary calculus. For example, for a random variable 
$Y \in {\mathbb{D}}^{1,p}$ and $g \in C^{1}(\mathbb{R}; \mathbb{R})$ with bounded derivative 
the chain rule reads as 
\begin{eqnarray*}
D g(Y)= g'(Y) \, DY
\end{eqnarray*}
This rule admits also a multidimensional localized version. Assume that 
\begin{itemize}
\item[(i)]  $g \in C^{1}(\mathbb{R}^d; \mathbb{R})$,
\item[(ii)] $Y_1, \ldots, Y_d \in \mathbb{D}^{1,p}$,
 \item[(iii)] $g(Y_1, \ldots, Y_d) \in L^p(\Omega)$,
 \item[(iv)] $ \nabla g(Y_1, \ldots, Y_d) \cdot (DY_1, \ldots, D Y_d) \in L^p(\Omega; L^2([0,T]; \mathbb{R}^2))$, 
\end{itemize}
then the chain rule also holds:  $g(Y) \in \mathbb{D}^{1,p}$ and its derivative is given by 
\begin{align}  \nabla g(Y_1, \ldots, Y_d) \cdot (DY_1, \ldots, D Y_d)
\label{chain_rule_md_2}
\end{align}

The divergence operator $\delta$ is the adjoint of the derivative operator. If a random variable $u \in L^{2} \big(\Omega;L^2([0,T]; \mathbb{R}^2) \big) $
belongs to $\textrm{dom}(\delta)$, the domain of the divergence operator, then  $\delta(u)$ is defined by the duality (also called integration by parts) relationship
\begin{eqnarray}\label{duality}
{ E} [Y \delta(u)]= { E} \left[ \int_0^T  \langle D_s Y,  u_s  \rangle \d s \right] \qquad \textrm{ for all } \quad Y \in {\mathbb D}^{1,2} \end{eqnarray}
If $u$ is adapted to the canonical filtration generated by $(W,B)$ 
and satisfies $E \int_0^T |u_t|^2  \d t <\infty$, then $u \in  \textrm{dom}(\delta)$ and
 $\delta(u)$ coincides with the It\=o integral $\int_0^T u_1(s) \d W_s + \int_0^T u_2(s) \d B_s$.

\bigskip
\subsection{Properties of the CIR process}
We will  need the following estimates for the CIR process, which are well known or can be found in \cite{HK}.

\begin{lemma} {\it \label{CIR_moments} 
(1) We have
$$ E \Big( \sup_{t \in [0,T]} V_t^p \Big) < \infty$$ for all $p \geq 1$ and
$$  \sup_{t \in [0,T]} E V_t^p  < \infty \qquad \textrm{iff} \qquad  p > -\frac{2 \kappa \lambda}{\theta^2}$$
(2) We have
$$ E \exp (p V_T) < \infty \qquad \textrm{iff} \qquad   p < \frac{2 \kappa}{\theta^2} \frac{1}{1-\exp(-\kappa T)} $$
(3) For all $p \geq 1$, there exist constants $C_p >0$ such that
$$ E|V_t -V_s|^p \leq C_p |t-s|^{p/2}, \qquad s,t \in [0,T] $$}
\end{lemma}

\bigskip

\subsection{Properties of the discretization scheme}

We also require several estimates for our discretization of the CIR process. For their and also the subsequent proofs we introduce the following notation:  For a fixed time discretization $0=t_0 < t_1 < \dots <t_N=T$, define $n(t):=\max\{n\in\{0,\dots,N\}  : t_n\leq t\}$, $\eta(t):= t_{n(t)}$ and $\Delta_t = t-\eta(t)$. Our proofs will make use of the following processes:
\begin{align*}
\widetilde{W}_t &:= W_t-W_{\eta(t)} \\
\widetilde{B}_t &:= B_t-B_{\eta(t)} \\
\widehat{x}_t &:= {x}_{n(t)}  -\frac12 {v}_{n(t)} \Delta_t       + \sqrt{{v}_{n(t)}} \Big(\rho \widetilde{W}_t + \sqrt{1-\rho^2} \widetilde{B}_t \Big) \\
\widetilde{v}_t &:= {v}_{n(t)} + \kappa\lambda \Delta_t + \theta\sqrt{{v}_{n(t)}} \widetilde{W}_t + \frac{\theta^2}{4} (\widetilde{W}^2_t-\Delta_t) \\
\widehat{v}_t &:= \frac{1}{1+\kappa\Delta_t} \widetilde{v}_t
\end{align*}
Note that $\lim_{t\nearrow t_n} \widehat{x}_t = \widehat{x}_{t_n} = {x}_n$ and $\lim_{t\nearrow t_n} \widehat{v}_t = \widehat{v}_{t_n} = {v}_n$ and that inside each interval $[t_n,t_{n+1}]$ the processes $\widehat{x}_t$ and $\widetilde{v}_t$ are It\=o processes:
\begin{align*}
\widehat{x}_t &:= {x}_{n(t)} -\frac12 \int_{\eta(t)}^t  {v}_{n(t)} \d s + \int_{\eta(t)}^t \sqrt{{v}_{n(t)}} \;\;\d \Big(\rho {W}_s + \sqrt{1-\rho^2} {B}_s\Big) \\
\widetilde{v}_t &:= {v}_{n(t)} + \int_{\eta(t)}^t \kappa\lambda \;\d s + \int_{\eta(t)}^t \Big(\theta\sqrt{{v}_{n(t)}} +\frac{\theta^2}{2} \widetilde{W}_s\Big) \d {W}_s
\end{align*}

\smallskip
The quantities on which numerical constants depend will be indicated by subscripts. In particular, constants will be independent of the discretization $\{ t_1, \ldots, t_N \}$ unless stated otherwise. 

\medskip

\begin{lemma}  \label{inverse_moments} Let (F-min) be satisfied. (1) For all $p \geq 1$ there exists a constant $C=C_{p, \kappa, \lambda, \theta, v_0, T} >0$ such that
$$ E \Big( \sup_{t \in [0,T]} \widehat{v}_t^p \Big) \leq C $$
(2) For all $p \leq \frac{2 \kappa \lambda}{\theta^2}-1$  there exists a constant $C=C_{p, \kappa, \lambda, \theta, v_0, T} >0$ such that
$$  \sup_{t \in [0,T]} E \big( \widehat{v}_t^{-p} \big) \leq C $$
(3) For all $p \leq \frac{ 2\kappa }{\theta^2}$ there exists a constant $C=C_{p, \kappa, \lambda, \theta, v_0, T} >0$ such that
$$  \sup_{t \in [0,T]} E \exp( p \widehat{v}_t \big) \leq C $$
(4) We have 
$$ {v}_{k} \geq   \frac{1}{1+\kappa T }\Big( \kappa \lambda - \frac{\theta^2}{4} \Big)  (t_k-t_{k-1})  \qquad  \textrm{for }  \quad k=   1, \ldots, N,    $$ 
and, respectively,
$$  \widehat{v}_{t} \geq  \frac{1}{1+\kappa T} \Big ( \kappa \lambda -\frac{\theta^2}{4} \Big) \Delta_t \qquad \textrm{for } \quad t \in [0,T] \setminus \{ t_0, \ldots, t_N \}  $$
(5)  For all $2 \leq q \leq \frac{4 \kappa \lambda}{\theta^2}-2$ there exists a constant $C=C_{q, \kappa, \lambda, \theta, v_0, T} >0$ such that
$$    \sup_{t \in [0,T]}    t^{-q/2}  E \left| \int_0^t  \frac{1}{\sqrt{\widehat{v}_{\eta(u)}}} \d B_u \right|^{q} \leq  C$$
(6) For all $q\geq 1$  there exists a constant $C=C_{q, \kappa, \lambda, \theta,  T} >0$ such that
$$  \sup_{t \in [0,T]} E \left| \frac{\widehat{v}_{\eta(t)}}{\widehat{v}_t} \right|^q \leq C$$

\end{lemma}
\begin{proof}
Assertion (1) can be shown by straightforward calculations using the Burkholder--Davis--Gundy inequality.

For assertion (2)  let $\varepsilon \in (0,v_0)$ and define $\tau_{\varepsilon}:= \inf \{t \geq 0: \widetilde{v}_t= \varepsilon\}.$ Applying It\=o's lemma,
noting that
$$ \widehat{v}_{\eta(t)}=  v_{n(t)}= \widetilde{v}_{\eta(t)} $$
and taking  expectations give
\begin{align*} E \big( \widetilde{v}_{t \wedge \tau_{\varepsilon}}^{-p} \big) = & E  \big(\widetilde{v}_{\eta(t) \wedge \tau_{\varepsilon}}^{-p} \big)  - p \kappa \lambda E \left(  \int_{\eta(t) \wedge \tau_{\varepsilon}}^{t \wedge \tau_{\varepsilon}} \widetilde{v}_u^{-p-1}\d u\right)
 \\ & \qquad + p(p+1)\frac{\theta^2}{2}  E \left(  \int_{\eta(t) \wedge \tau_{\varepsilon}}^{t \wedge \tau_{\varepsilon}} \widetilde{v}_u^{-p-2}\Big( \sqrt{\widehat{v}_{\eta(u)}} + \frac{\theta}{2} \widetilde{W}_u  \Big)^2\d u\right), \qquad t \in [t_{n(t)},t_{n(t)+1}]
\end{align*}
However, since
\begin{align} \widetilde{v}_t = \Big( \sqrt{\widehat{v}_{\eta(t)}} + \frac{\theta}{2} \widetilde{W}_t  \Big)^2 +  \Big( \kappa \lambda - \frac{\theta^2}{4}\Big)\Delta_t,  \label{vtilde} \end{align}
it follows
$$ \widetilde{v}_u^{-p-2}\Big( \sqrt{\widehat{v}_{\eta(u)}} + \frac{\theta}{2} \widetilde{W}_u  \Big)^2 \leq \widetilde{v}_u^{-p-1} ,$$ 
 and thus we have
 \begin{align*} E \big( \widetilde{v}_{t \wedge \tau_{\varepsilon}}^{-p} \big) \leq E  \big( \widetilde{v}_{\eta(t) \wedge \tau_{\varepsilon}}^{-p} \big)  + p \left( (p+1)\frac{\theta^2}{2} -  \kappa \lambda \right) E \left(  \int_{\eta(t) \wedge \tau_{\varepsilon}}^{t \wedge \tau_{\varepsilon}} \widetilde{v}_u^{-p-1}\d u\right)
\end{align*}
Now $2 \kappa \lambda / \theta^2 \geq p +1$ implies 
\begin{align}  E \big( \widetilde{v}_{t \wedge \tau_{\varepsilon}}^{-p} \big) \leq  E  \big(\widetilde{v}_{\eta(t) \wedge \tau_{\varepsilon}}^{-p} \big) 
= E  \big(\widehat{v}_{\eta(t) \wedge \tau_{\varepsilon}}^{-p} 1_{\{ \tau_{\varepsilon}  \geq \eta(t) \}} \big) + E     \big(\widetilde{v}_{\tau_{\varepsilon}}^{-p} 1_{\{ \tau_{\varepsilon}  < \eta(t) \}} \big)
\label{help_1}
\end{align}  
Let $t < t_1$. Then $\eta(t)=0$ and $$ E \big(\widetilde{v}_{\eta(t) \wedge \tau_{\varepsilon}}^{-p} \big) = v_0^{-p}, $$ thus \eqref{help_1} implies
 $$  \sup_{t \in [0,t_1]} E \big( \widetilde{v}_{t \wedge \tau_{\varepsilon}}^{-p} \big) \leq v_0^{-p} $$
 Hence we have
 $$  \sup_{t \in [0,t_1]} E \big( \widehat{v}_{t \wedge \tau_{\varepsilon}}^{-p} \big) \leq \exp(p\kappa t_1) v_0^{-p} $$
An induction over the discretization subintervals using \eqref{help_1}  now yields
$$   \sup_{t \in [0,T]}  E \big( \widehat{v}_{t \wedge \tau_{\varepsilon}}^{-p} \big)  \leq \exp(p\kappa T) v_0^{-p}    $$
and an application of Fatou's lemma concludes the proof for $\varepsilon \rightarrow 0$.

To prove assertion (3) let $\varepsilon \in (0, v_0^{-1})$ and define $\tau_{\varepsilon}:= \inf \{t \geq 0: \widehat{v}_t= \varepsilon^{-1} \}.$ Applying It\=o's lemma to $(t,v) \mapsto \exp(p \frac{1}{1+ \kappa \Delta_t}v)$ and taking  expectations give
\begin{align*} E \exp(p \widehat{v}_{t \wedge \tau_{\varepsilon}})  = &  E  \exp( p \widehat{v}_{\eta(t) \wedge \tau_{\varepsilon}} \big) \\
& + p \kappa \lambda E \left(  \int_{\eta(t) \wedge \tau_{\varepsilon}}^{t \wedge \tau_{\varepsilon}} \frac{\exp(p  \widehat{v}_u)}{1+ \kappa \Delta_u}
 \d u\right)
 \\ & + p^2\frac{\theta^2}{2}  E \left(  \int_{\eta(t) \wedge \tau_{\varepsilon}}^{t \wedge \tau_{\varepsilon}} \frac{\exp(p\widehat{v}_u)}{(1+ \kappa \Delta_u)^2} \Big( \sqrt{\widehat{v}_{\eta(u)}} + \frac{\theta}{2} \widetilde{W}_u  \Big)^2\d u\right)
  \\ & - p \kappa  E \left(   \int_{\eta(t) \wedge \tau_{\varepsilon}}^{t \wedge \tau_{\varepsilon}} \frac{\exp(p\widehat{v}_u) \widehat{v}_{u}}{1+ \kappa \Delta_u}   \d u\right),  \qquad t \in [t_{n(t)},t_{n(t)+1}]
 \end{align*}
Recall  that
$$ \widehat{v}_t = \frac{1}{1+ \kappa \Delta_t}\Big( \sqrt{\widehat{v}_{\eta(t)}} + \frac{\theta}{2} \widetilde{W}_t  \Big)^2 +   \frac{1}{1+ \kappa \Delta_t} \Big( \kappa \lambda - \frac{\theta^2}{4}\Big)\Delta_t,$$
and thus $p \leq \frac{2 \kappa}{\theta^2}$ implies that
 \begin{align*}  E \exp(p \widehat{v}_{t \wedge \tau_{\varepsilon}})  \leq  &  E  \exp( p \widehat{v}_{\eta(t) \wedge \tau_{\varepsilon}} \big) 
+ p \kappa \lambda   \int_{\eta(t)}^{t  } E \exp(p  \widehat{v}_{u \wedge \tau_{\varepsilon}} )  \d u
\end{align*}
Gronwall's Lemma now yields
$$ E \exp(p \widehat{v}_{t \wedge \tau_{\varepsilon}})  \leq   E  \exp( p \widehat{v}_{\eta(t) \wedge \tau_{\varepsilon}} \big) \exp (\kappa \lambda p \Delta_t) $$
An induction over the discretization subintervals gives
$$ E \exp(p \widehat{v}_{t \wedge \tau_{\varepsilon}})    \leq   \exp( p v_0) \exp (\kappa \lambda p T) $$
and an application of Fatou's lemma concludes the proof for $\varepsilon \rightarrow 0$.

Assertion (4) is a consequence of 
\begin{align*} \widehat{v}_t & =  \frac{1}{1+ \kappa \Delta_t}\left(\sqrt{ \widehat{v}_{\eta(t)}} + \frac{\theta}{2} \widetilde{W}_t \right)^2 + \frac{1}{1+ \kappa \Delta_t} \Big(  \kappa \lambda- \frac{\theta^2}{4} \Big)\Delta_t
\\ & \geq \frac{1}{1+ \kappa T } \Big(  \kappa \lambda- \frac{\theta^2}{4} \Big)\Delta_t
 \end{align*} for $t >0$.

Assertion (5) follows straightforwardly from   (2) and the Burkholder--Davis--Gundy inequality.

For assertion (6) note that it is enough to show that
$$ \sup_{t \in [0,T]} E \left| \frac{\widehat{v}_{\eta(t)}}{\widetilde{v}_{t} } \right|^p \leq C .$$
However, \eqref{vtilde} and the independence of $\widetilde{W}_t$ and $\widehat{v}_{\eta(t)}$ imply that
$$  E \left( \Big| \frac{\widehat{v}_{\eta(t)}}{\widetilde{v}_{t}} \Big|^p  \Big | \widehat{v}_{\eta(t)}= \xi \right) = 
 E \left| \frac{\xi}{(\sqrt{\xi}+  \frac{\theta}{2} \widetilde{W}_t)^2 + c \Delta_t} \right|^p $$
 where   $c=  \kappa \lambda- \frac{\theta^2}{4}$.  
 Now set
 $$ A=  \{ \xi - \theta^{2}  \widetilde{W}_t^2 \geq 0 \}.$$ 
 Since $(a-b)^2 \geq \frac{1}{2}a^2 -b^2$ it follows
 $$ \frac{\xi}{(\sqrt{\xi}+  \frac{\theta}{2} \widetilde{W}_t )^2 + c \Delta_t}  {1}_A  \leq \frac{\xi}{\frac{\xi}{2} -  \frac{\theta^2}{4}  \widetilde{W}_t^2+ c \Delta_t}  {1}_A $$
Now, on $A$ we have 
$$ \frac{\xi}{2} -  \frac{\theta^2}{4}  \widetilde{W}_t^2 \geq \frac{\xi}{4},$$ and we obtain 
 $$  E \left( \left| \frac{\xi}{(\sqrt{\xi}+  \frac{\theta}{2} \widetilde{W}_t)^2 + c \Delta_t} \right|^p {1}_{A} \right)
 \leq  \left| \frac{ \xi}{ \frac{\xi}{4} + c \Delta_t} \right|^p P(A)\leq 4^p$$     
Moreover, on the complementary event we have 
$$ E \left( \left| \frac{\xi}{(\sqrt{\xi}+  \frac{\theta}{2} \widetilde{W}_t)^2 + c \Delta_t} \right|^p {1}_{ \Omega \setminus A } \right)
\leq 2 \left| \frac{\xi}{ c \Delta_t} \right|^p  P \left( W_1 >  \sqrt{ \frac{1}{\theta^2} \frac{\xi}{ \Delta_t }} \right)
$$
Using a standard tail estimate for the Gaussian distribution, i.e.
$$ P(W_1>x) \leq \frac{\exp \left( - x^2 /2 \right)}{x \sqrt{2 \pi}}, \qquad x>0,$$
it follows
$$ E \left( \left| \frac{\xi}{(\sqrt{\xi}+  \frac{\theta}{2} \widetilde{W}_t)^2 + c \Delta_t} \right|^p {1}_{ \Omega \setminus A } \right)
\leq  C \left| \frac{ \xi}{ 2 \theta^2 \Delta_t} \right|^{p-1/2}  \exp\left( - \frac{1}{2\theta^2} \frac{\xi}{\Delta_t} \right)
$$ for some constant $C=C_{p,\kappa, \lambda, \theta} >0$.
But we have
$$ \sup_{y \geq 0} \, y^{p-1/2} \exp(-y) \leq (p-1/2)^{p-1/2} \exp(-p+1/2),  $$
for $p \geq 1$,
and 
therefore
$$ E \left( \left| \frac{\xi}{(\sqrt{\xi}+  \frac{\theta}{2} \widetilde{W}_t)^2 + c \Delta_t} \right|^p {1}_{ \Omega \setminus A } \right) \leq C (p-1/2)^{p-1/2} \exp(-p+1/2)$$
So finally,  we can conclude that there exists a constant $C=C_{p, \kappa, \lambda, \theta}>0$ such that
 $$ E \left( \Big| \frac{\widehat{v}_{\eta(t)}}{\widetilde{v}_{t}} \Big|^p  \Big | \widehat{v}_{\eta(t)}= \xi \right)  \leq  C, $$
which implies that
$$ \sup_{t \in [0,T]} E \Big| \frac{\widehat{v}_{\eta(t)}}{\widetilde{v}_{t}} \Big|^p   \leq  C$$ 
\end{proof}

\bigskip
By  straightforward computations and using the first assertion of the previous Lemma, we also have:
\begin{lemma}  \label{smooth_inc} {\it (1) For all $p \geq 1$, there exists a constant $C=C_{p, \kappa, \lambda, \theta, v_0, T} >0$ such that
$$  E | \widehat{x}_t - \widehat{x}_s|^p \leq C \cdot |t-s|^{p/2}, \qquad s,t \in [0,T]$$
(2) For all $p \geq 1$, there exists a constant $C=C_{p, \kappa, \lambda, \theta, v_0, T} >0$ such that
$$  E | \widehat{v}_t - \widehat{v}_s|^p \leq C \cdot |t-s|^{p/2}, \qquad s,t \in [0,T]$$}
\end{lemma}

\bigskip
The next lemma deals with the Malliavin smoothness of our approximation of the log-Heston SDE. Here we use the notation ${\mathbb D}^{1,\infty}= \cap_{p \geq 1} {\mathbb D}^{1,p}$.

\medskip

\begin{lemma} \label{mall-smooth} {\it Let $t \in [0,T]$. Under (F-min) we have
$\widehat{x}_t,\widehat{v}_t \in {\mathbb D}^{1,\infty}$. In particular
$$ D_r^B  \widehat{x}_t = \sqrt{1-\rho^2}\sqrt{ \widehat{v}_{\eta(r)}} 1_{[0,t]}(r), \qquad r,t \in [0,T]$$}
\end{lemma}
\begin{proof}
 (1) We consider first the discretized volatility process. For a fixed discretization $0=t_0<t_1< \ldots < t_N=T$ Lemma \ref{inverse_moments} (4) implies  the existence of a constant $C=C_{\kappa, \lambda, \theta, v_0, t_1, \ldots, t_N}>0$ such that
 $$ \inf_{t \in [0,T]} \widehat{v}_{\eta(t)} \geq  C  $$
 Hence we can write
 \begin{align*} \widehat{v}_t & =  \frac{1}{1+ \kappa \Delta_t}\left(g(\widehat{v}_{\eta(t)}) + \frac{\theta}{2} \widetilde{W}_t \right)^2 + \frac{1}{1+ \kappa \Delta_t} \Big(  \kappa \lambda- \frac{\theta^2}{4} \Big)\Delta_t
 \end{align*} 
 where $g \in C^{1}(\mathbb{R};\mathbb{R})$ with bounded derivative and $g(x)=\sqrt{x}$ for $x \geq C/2$. Now fix $t>0$ and assume that $ \widehat{v}_{\eta(t)} \in {\mathbb D}^{1,\infty}$. Then, the 
 localised chain rule implies that $ \widehat{v}_{t} \in {\mathbb D}^{1,\infty}$, since
  $$D^B \widehat{v}_t=0,$$ 
 due to the independence of
  $W$ and $B$, 
 and 
\begin{align*}
  D_r^W \widehat{v}_t
   &= \frac{2}{1+ \kappa \Delta_t} \Big(g( \widehat{v}_{\eta(t)}) + \frac{\theta}{2} \widetilde{W}_t \Big) \Big( g'( \widehat{v}_{\eta(t)}) D_r^W \widehat{v}_{\eta(t)} + \frac{\theta}{2}  1_{(\eta(t),t]}(r) \Big) 
  \end{align*} 
 by the chain rule \eqref{chain_rule_md_2} and using the boundedness of $g'$ as well as the existence of all moments of $ \sup_{t\in [0,T]} \widehat{v}_t$.
Now, ${v}_0$ is non-random, so we obtain  $\widehat{v}_{t} \in {\mathbb D}^{1,\infty}$  by   induction.

(2) Note that
$$ \widehat{x}_t = \widehat{x}_{\eta(t)} - \frac{1}{2} \widehat{v}_{\eta(t)} (t - \eta(t)) + \rho  g(\widehat{v}_{\eta(t)}) \widetilde{W}_t + \sqrt{1-\rho^2}  g(\widehat{v}_{\eta(t)}) \widetilde{B}_t $$
and 
$$  \widehat{x}_{\eta(t)}= - \frac{1}{2} \sum_{k=0}^{n(t)-1} {v}_k (t_ {k+1}-t_{k}) +  \sum_{k=0}^{n(t)-1} g({v}_k) \Big( \rho ( W_{t_{k+1}}-W_{t_{k}} ) 
+ \sqrt{1-\rho^2} ( B_{t_{k+1}}-B_{t_{k}} ) \Big)
$$
Thus, a direct application of the localised chain rule and the first step give that
$  \widehat{x}_t \in \mathbb{D}^{1,\infty}$ for any $t \in [0,T]$. Moreover, since $D_r^B v_k= D_r^Bg({v}_k)= D_r^B( W_{t_{k+1}}-W_{t_{k}} )=0$ and $g({v}_k)=\sqrt{{v}_k }$ the chain rule also yields
$$ D_r^B  \widehat{x}_{\eta(t)}=   
  \sqrt{1-\rho^2} \sum_{k=0}^{n(t)-1} \sqrt{{v}_k}  1_{({t_{k}},t_{k+1}]}(r)  $$
  and $$  D_r^B  \widehat{x}_{t}= D_r^B  \widehat{x}_{\eta(t)} +  \sqrt{1-\rho^2}  \sqrt{ {v}_{n(t)}}  1_{({t_{n(t)}},t_{n(t)+1}]}(r) $$

\end{proof}

\subsection{Drift-implicit square-root Euler approximation of CIR}
A helpful tool for the proof of Theorem \ref{theorem-weak} will be the  so called  drift-implicit square-root Euler approximation of the CIR process proposed by Alfonsi \cite{AA1}.
This scheme reads as
\begin{align} \label{sqrt_euler} \begin{split}
 a_{k+1} &=  \left( \frac{\sqrt{a_{k}} +   \frac{\theta}{2} \Delta_k W}{2 +\kappa (t_{k+1}-t_k)} +  \sqrt{ \frac{(\sqrt{a_{k}} +  \frac{\theta}{2} \Delta_k W)^2}{(2+ {\kappa} (t_{k+1}-t_k)^2} + \frac{ ( \kappa \lambda - \frac{\theta^2}{4})(t_{k+1}-t_k)}{2 + {\kappa} (t_{k+1}-t_k)} } \right)^2,  \\
 a_0 & = v_0, \end{split}
\end{align}
 and is well defined and positive under {\it (F-min)}, i.e.\ $\frac{4 \kappa \lambda}{\theta^2} \geq 1$. It arises by discretizing the Lamperti-transformed process $A_t=\sqrt{V_t}$, $t \in [0,T]$,
with  a drift-implicit Euler scheme, and transforming back.

Strong convergence rates for this scheme have been established for $\frac{2 \kappa \lambda}{\theta^2} >1$ in
\cite{DNS,AA3,NS}.
The recent work \cite{HJN} performs a convergence analysis under {\it (F-min)}. The authors establish $L^p$-convergence rates for \eqref{sqrt_euler} in the case of an equidistant discretization.  Using Corollary 3.9 in \cite{HJN} and Lemma \ref{CIR_moments} (1) and (3) we obtain $L^1$-convergence 
 without a rate for general discretizations,  i.e.\ it holds
 \begin{align} \label{hjn_help}\lim_{\Delta \rightarrow 0} \,  E \sup_{k=0, \ldots, N} |\sqrt{a_k} -\sqrt{V_{t_k}} |=0  \end{align}
 under {\it (F-min)}.

\medskip

Note that the drift-implicit Milstein scheme dominates the square-root Euler approximation:
\begin{align} v_k \geq a_k, \qquad k=0,1,2, \ldots \label{dom_sqrt} \end{align} To see this, set
\begin{align*}
 a_{k+1}^x=   \left( \frac{\sqrt{x} +   \frac{\theta}{2} \Delta_k W}{2 +\kappa (t_{k+1}-t_k))} +  \sqrt{ \frac{(\sqrt{x} +  \frac{\theta}{2} \Delta_k W)^2}{(2+ {\kappa} (t_{k+1}-t_k))^2} + \frac{ ( \kappa \lambda - \frac{\theta^2}{4})(t_{k+1}-t_k)}{2 + {\kappa} (t_{k+1}-t_k)} } \right)^2,
\end{align*}
and
\begin{align*}
\qquad {v}_{k+1}^x &= x + \kappa(\lambda-v_{k+1}^x) (t_{k+1}-t_k) + \theta \sqrt{x} \Delta_n W + \frac{\theta^2}{4} \big((\Delta_k W)^2 - (t_{k+1}-t_k)   \big)
\end{align*} with $x \geq 0$.
From \cite{AA1} it is known that $a_{k}^x$ is increasing in $x$ for all $x \geq 0$, $k\in \mathbb{N}$.
Since 
$$a_{k+1}^x=v_{k+1}^x -  \frac{1}{1+\kappa (t_{k+1}-t_k)} \left(  \frac{4 \kappa \lambda - \theta^2}{8 \sqrt{a_{k+1}^x}} - \frac{\kappa}{2} \sqrt{a_{k+1}^x} \right)^2 (t_{k+1}-t_k)^2, $$
an induction gives \eqref{dom_sqrt}.

Using this domination property and Lemma \ref{inverse_moments} (1) we obtain
\begin{align} \label{bound_ak} E \sup_{k=0, \ldots, N} |a_k|^p < \infty \end{align}
for all $p \in \mathbb{N}$. Since moreover
$$ | a_k - V_{t_k}|^p =   | \sqrt{a_k} - \sqrt{V_{t_k}}|^p  \cdot |\sqrt{a_k} + \sqrt{V_{t_k}}|^p \leq |\sqrt{a_k} - \sqrt{V_{t_k}}|^{1/(1+\varepsilon)}  \cdot    |\sqrt{a_k} + \sqrt{V_{t_k}}|^{2p - 1/(1+ \varepsilon)}$$
we have
$$  | a_k - V_{t_k}|^p \leq   |\sqrt{a_k} - \sqrt{V_{t_k}}|^{1/(1+\varepsilon)}  \cdot C_{p,\varepsilon} \big( 1 + a_k^{2p - 1/(1+ \varepsilon)}  + V_{t_k}^{2p - 1/(1+ \varepsilon)} \big)  $$ 
for some constant $C_{p,\varepsilon}>0$. Now
estimates \eqref{hjn_help}, \eqref{bound_ak}, Lemma \ref{CIR_moments}  (1)  and H\"older's inequality  give
\medskip

\begin{lemma}\label{con_sqrt}
 Let $p \geq 1$. Under {\it (F-min)}, we have
 $$  \lim_{\Delta \rightarrow 0}  \, \sup_{k=0, \ldots, N}  E   |a_k -V_{t_k} |^p = 0 $$
\end{lemma}

\medskip

\section{Proof of the Main Results}

\subsection{Proof of Theorem \ref{theorem-main}}

Following \cite{TT} we write the weak error as telescoping sum of local errors, i.e.
$$|E(h({x}_N,v_N))-E(h(X_T,V_T))| = \left|\sum_{n=1}^N E\big(u(t_n, {x}_n, {v}_n) - u(t_{n-1},{x}_{n-1}, {v}_{n-1})\big)\right|$$ 
where $h(x,v)=f(x) \phi_M(v)$ with $f$ satisfying $(S)$ and the localizing function $\phi_M$ from Theorem \ref{FP-main}.

Next we expand the local errors using the It\=o formula and the function $$ \widetilde{u}(t,x,v):=u(t,x,v/(1+\kappa\Delta_t)), \qquad t\in [0,T], x \in \mathbb{R}, v \geq 0$$ For brevity we will often omit the arguments of $\widetilde{u}(t,\widehat{x}_t, \widetilde{v}_t)$ and $u(t,\widehat{x}_t,\widehat{v}_t)$.
We have
\begin{align*}
e_n&:=E\big(u(t_{n+1}, {x}_{n+1}, {v}_{n+1}) - u(t_n,{x}_n, {v}_n)\big) \\
&= E\big(\widetilde{u}(t_{n+1},\widehat{x}_{t_{n+1}}, \widetilde{v}_{t_{n+1}}) - \widetilde{u}(t_n,\widehat{x}_{t_n}, \widetilde{v}_{t_n})\big) \\
&= \int_{t_n}^{t_{n+1}} E\left[ \partial_t \widetilde{u}(t,\widehat{x}_t, \widetilde{v}_t) - \frac12 {v}_n \partial_x \widetilde{u} + \kappa\lambda \partial_v \widetilde{u} + \frac12 {v}_n \partial_{xx} \widetilde{u}\right.\\
&\qquad+\left.\sqrt{{v}_n} \rho\theta \left(\sqrt{{v}_n}+\frac{\theta}{2} \widetilde{W}_t\right) \partial_{xv}\widetilde{u} + \frac{\theta^2}{2} \left(\sqrt{{v}_n} + \frac{\theta}{2}\widetilde{W}_t\right)^2 \partial_{vv} \widetilde{u} \right] \; \d t
\end{align*}
The derivatives of $\widetilde{u}$ can be written in terms of derivatives of $u$:
\begin{align*}
 \partial_t\widetilde{u}(t,\widehat{x}_t,\widetilde{v}_t) &= \partial_t u(t,\widehat{x}_t,\widehat{v}_t) - \frac{\kappa \widehat{v}_t}{1+\kappa\Delta_t}\cdot \partial_v u(t,\widehat{x}_t,\widehat{v}_t)  \\
\frac{\partial^{k+l}}{\partial x^k\partial v^l} \widetilde{u}(t,\widehat{x}_t,\widetilde{v}_t) &= \frac{1}{(1+\kappa\Delta_t)^l} \cdot \frac{\partial^{k+l}}{\partial x^k\partial v^l} u(t,\widehat{x}_t,\widehat{v}_t)
\end{align*}

Using $(\sqrt{{v}_n}+ \frac{\theta}{2}\widetilde{W}_t)^2 = \widetilde{v}_t - (\kappa\lambda-\theta^2/4)\Delta_t$ and the Kolmogorov-backward PDE for $u$, i.e.
$$\partial_t u = \frac12 v \partial_x u - \kappa(\lambda-v) \partial_v u - \frac12 v \partial_{xx} u - \rho\theta v \partial_{xv} u  - \frac{\theta^2}{2} v \partial_{vv} u,$$
we can write the local error expansion as
\begin{align*}
  e_n &= \int_{t_n}^{t_{n+1}} E\left[ \frac12 (\widehat{v}_t - {v}_n) \partial_x u + \kappa\lambda \left(\frac{1}{1+\kappa\Delta_t}-1\right) \partial_v u + \kappa\left(1-\frac{1}{1+\kappa\Delta_t}\right) \widehat{v}_t \partial_v u \right. \\
&\qquad + \frac12 \left({v}_n - \widehat{v}_t\right) \partial_{xx} u + \rho\theta \left(\frac{{v}_n}{1+\kappa\Delta_t} - \widehat{v}_t\right) \partial_{xv} u + \frac{\theta^2\rho}{2} \sqrt{{v}_n} \widetilde{W}_t \frac{1}{1+\kappa\Delta_t} \partial_{xv} u \\
&\left.\qquad + \frac{\theta^2}{2}\left(\frac{1}{1+\kappa\Delta_t}-1\right)\widehat{v}_t \partial_{vv} u - \frac{\Delta_t\theta^2}{2(1+\kappa\Delta_t)^2} \left(\kappa\lambda-\frac{\theta^2}{4}\right) \partial_{vv}u
 \right] \d t
\end{align*}
In the next step we use the identities
\begin{align*}
\frac{{v}_n}{1 + \kappa \Delta_t} - \widehat{v}_t &=  \frac{1}{1+ \kappa \Delta_t} ({v}_n - \widetilde{v}_t),\\
\widehat{v}_t-{v}_n & = \widetilde{v}_t-{v}_n-\kappa\Delta_t \widehat{v}_t = \kappa\Delta_t(\lambda-\widehat{v}_t) + \theta \sqrt{{v}_n}\widetilde{W}_t + \frac{\theta^2}{4}(\widetilde{W}_t^2 - \Delta_t),
\end{align*}
and after regrouping the terms we end up with
\begin{align*}
e_n &= e_n^{(1)} + e_n^{(2)} + e_n^{(3)},
\end{align*}
where
\begin{align*}
e_n^{(1)}&=\int_{t_n}^{t_{n+1}} \Delta_t \cdot E\left[\frac{\kappa^2}{1+\kappa\Delta_t} (\widehat{v}_t-\lambda) \partial_v u - \frac{\theta^2}{2(1+\kappa\Delta_t)} \left(\kappa \widehat{v}_t + \frac{4\kappa\lambda- \theta^2}{4(1+\kappa\Delta_t)} \right) \partial_{vv} u\right.\\
&\qquad\qquad\qquad\qquad+\left. \frac{\kappa}{2} (\lambda-\widehat{v}_t) (\partial_x u -\partial_{xx} u) - \frac{\rho\theta\kappa\lambda}{1+\kappa\Delta_t} \partial_{xv} u\right] \d t, \\
e_{n}^{(2)}&= \int_{t_n}^{t_{n+1}} E\left[\sqrt{{v}_n} \widetilde{W}_t \left(\frac{\theta}{2} \partial_x u - \frac{\theta}{2} \partial_{xx} u - \frac{\rho\theta^2}{2(1+\kappa\Delta_t)} \partial_{xv} u \right)\right] \d t,\\
e_n^{(3)}&= \int_{t_n}^{t_{n+1}} E\left[ (\widetilde{W}_t^2 - \Delta_t) \cdot \left(\frac{\theta^2}{8} \partial_x u - \frac{\theta^2}{8} \partial_{xx} u - \frac{\theta^3\rho}{4(1+\kappa\Delta_t)} \partial_{xv} u \right)\right] \d t
\end{align*}

Now Theorem \ref{FP-main} implies that 
\begin{align} \label{deriv-1}
| \partial _t u(t,x, v)| +   | \partial_x u(t,x, v)| +  |\partial_v u(t,x, v)|  \leq  c(f,\varepsilon,q)(1+M^q), \qquad t \in [0,T], \, x \in \mathbb{R}, \, v \geq 0 \end{align}
 and
\begin{align}
 \label{deriv-2} & | \partial_{xx} u(t,x, v)| +  | \partial_{xv} u(t,x, v)|  +  | \partial_{vv} u(t,{x}, v)|  \\ & \qquad \qquad \qquad \qquad \qquad \qquad \qquad  \leq  c(f,\varepsilon,q)(1+M^q) \Big (1+ \frac{1}{{v}} \Big), \qquad t \in [0,T], \, x \in \mathbb{R}, \, v > 0 \nonumber
\end{align}
 In the following we denote by $c$ constants, which only depend on $c(f,\varepsilon,q)$, $\kappa$, $\lambda$, $\theta$, $\rho$, $T$, $x_0$, $v_0$ regardless of their value.
Using equations \eqref{deriv-1} and \eqref{deriv-2} we obtain
\begin{align*} 
 |e_n^{(1)}| \leq c  \left( \Delta^2 + E \int_{t_n}^{t_{n+1}}  \Big( \widehat{v}_t + \frac{1}{\widehat{v}_t}  \Big) \Delta_t \d t \right)  (1+M^q)
\end{align*}
and
\begin{align*} 
 |e_n^{(3)}| \leq c  \left( \Delta^2 +  E \int_{t_n}^{t_{n+1}}   \frac{1}{\widehat{v}_t} |\widetilde W_t^2 - \Delta_t|  \d t \right)(1+M^q)
\end{align*}
Since
$$ \sup_{t \in [0,T]} \left( E \left| \widehat{v}_t \right|^{p}\right)^{1/p}  +  \sup_{t \in [0,T]} \left( E \left| \frac{1}{\widehat{v}_t}  \right|^{1+\delta}\right)^{1/(1+\delta)} \leq c$$
for all $p \geq 1$ and $\delta \in \left(0, \frac{2\kappa \lambda}{\theta^2}-2 \right)$
by Lemma \ref{inverse_moments} (1), (2), we have
\begin{align} \label{en_1_final}
 |e_n^{(1)}|+ |e_n^{(3)}| &\leq c (1+M^q) \Delta^2
\end{align}

To deal with $e_n^{(2)}$ we will carry out an integration by parts first, which is summarized in the following lemma. Estimating this term directly would only give a bound of order $\sqrt{\Delta}$.

\medskip

\begin{lemma} \label{lem:ibp} {\it {
Let $t > 0$,  $g \in C^{(0,1,1)} ( [0,T] \times \mathbb{R} \times [0, \infty); \mathbb{R})$ be bounded and such that
$$   \int_0^T E\left | D_r \left( \sqrt{\widehat{v}_{\eta(t)}} \widetilde{W}_t g(t,\widehat{x}_t, \widehat{v}_t)  \right) \right |^2  \d r 
< \infty$$
 Then we have
 $$ E\left[\sqrt{\widehat{v}_{\eta(t)}} \widetilde{W}_t \partial_x g(t,\widehat{x}_t, \widehat{v}_t) \right] =   \frac{1}{ t \sqrt{1-\rho^2} } E\left[ \sqrt{\widehat{v}_{\eta(t)}} \widetilde{W}_t  g(t,\widehat{x}_t, \widehat{v}_t)
 \int_0^t \frac{1}{\sqrt{\widehat{v}_{\eta(r)}}} \d B_r
 \right]$$}}
\end{lemma}
\begin{proof}{
Because $\widetilde{W}$ and  $\widehat{v}$ are independent of $B$, the chain rule of Malliavin calculus implies  that $$D_r^{B}\Big( \sqrt{\widehat{v}_{\eta(t)}} \widetilde{W}_t  g(t,\widehat{x}_t, \widehat{v}_t) \Big)  =  \sqrt{\widehat{v}_{\eta(t)}}  \widetilde{W}_t \partial_x g(t,\widehat{x}_t,\widehat{v}_t) \cdot D_r^B \widehat{x}_t$$ with $D_r^B \widehat{x}_t = \sqrt{1-\rho^2}\sqrt{\widehat{v}_{\eta(r)}} {1}_{[0,t]}(r)$, see Lemma \ref{mall-smooth}.
Applying the integration by parts formula \eqref{duality} to 
$$ D_r Y= \left( \begin{array}{c} D^W_r \big( \sqrt{\widehat{v}_{\eta(t)}} \widetilde{W}_t  g(t,\widehat{x}_t, \widehat{v}_t) \big) \\ D^B_r \big( \sqrt{\widehat{v}_{\eta(t)}} \widetilde{W}_t  g(t,\widehat{x}_t, \widehat{v}_t) \big) \end{array} \right), \qquad u_r =
\left( \begin{array}{c} 0 \\ \frac{1}{\sqrt{\widehat{v}_{\eta(r)}}}{1}_{[0,t]}(r) \end{array} \right), \qquad r \in [0,T],
$$
we obtain
 \begin{align*}
  E \left( \sqrt{\widehat{v}_{\eta(t)}}  \widetilde{W}_t \partial_x g (t,\widehat{x}_t,\widehat{v}_t)  \right) &= \frac{1}{t \sqrt{1- \rho^2}}  E \left( \int_0^{t} D_r^B \left( \sqrt{\widehat{v}_{\eta(t)}}  \widetilde{W}_t g(t,\widehat{x}_t,\widehat{v}_t) \right) \frac{1}{\sqrt{\widehat{v}_{\eta(r)}}} \d r \right)
  \\ & =   \frac{1}{t \sqrt{1- \rho^2}} E \left[ \sqrt{\widehat{v}_{\eta(t)}} \widetilde{W}_t g(t,\widehat{x}_t,\widehat{v}_t)  \int_0^t \frac{1}{\sqrt{\widehat{v}_{\eta(r)}}} \d B_r \right]
  \end{align*}}
\end{proof}

\bigskip {
Now set
$$g(t,x,v)= \frac{\theta}{2}  u (t,x,v)- \frac{\theta}{2} \partial_{x} u(t,x,v) - \frac{\rho\theta^2}{2(1+\kappa\Delta_t)} \partial_{v} u (t,x,v)
$$
Theorem 3.1 implies that $g$ is bounded  and also provides the required smoothness assumptions for $g$. Moreover, the estimates \eqref{deriv-1} and \eqref{deriv-2} imply that
\begin{align}  | \partial_x g(t,x,v)| +|\partial_v g(t,x,v)|  \leq c \Big(1+ \frac{1}{v}\Big)(1+M^q), \qquad t \in [0,T], \, x \in \mathbb{R}, \,  v  \geq 0        \label{est-g}  \end{align}
Recall that $$\inf_{t \in [0,T]} \widehat{v}_{\eta(t)}  \geq  C$$ 
for some constant  $C=C_{\kappa, \lambda, \theta, v_0, t_1, \ldots, t_N}>0$ by Lemma \ref{inverse_moments} (4). Hence 
the  assumption of  Lemma  \ref{lem:ibp} is a consequence of Lemma \ref{mall-smooth}, Lemma \ref{inverse_moments} (1) and the Malliavin chain rule. Thus we can write
$$ e_n^{(2)} =\int_{t_n}^{t_{n+1}}  \frac{1}{t \sqrt{1- \rho^2}}  E \left[ \sqrt{\widehat{v}_{\eta(t)}} \widetilde{W}_t g(t,\widehat{x}_t,\widehat{v}_t)I^B_t   \right]  dt$$
with
$$I^B_t= \int_0^t \frac{1}{\sqrt{\widehat{v}_{\eta(r)}}} \d B_r, \qquad t \in [0,T]$$
Since moreover $ \widetilde{W}_t$ is independent  of $B,\widehat{x}_{\eta(t)}$ and $\widehat{v}_{u}, u \in [0,\eta(t)]$, it follows that
$$  E \left[ \sqrt{\widehat{v}_{\eta(t)}} \widetilde{W}_t g(t,\widehat{x}_{\eta(t)},\widehat{v}_{\eta(t)}) I_t^B\right] = 0 $$ 
and hence
\begin{align*}
e_n^{(2)}  &=
 \int_{t_n}^{t_{n+1}}  \frac{1}{t \sqrt{1- \rho^2}}  E \left[ \sqrt{\widehat{v}_{\eta(t)}} \widetilde{W}_t I_t^B \big( g(t,\widehat{x}_t,\widehat{v}_t) - g(t,\widehat{x}_{\eta(t)},\widehat{v}_{\eta(t)})\big)\right]  dt
\end{align*} 
The mean value theorem  now gives 
\begin{align*}
g(t,\widehat{x}_t,\widehat{v}_t) - g(t,\widehat{x}_{\eta(t)},\widehat{v}_{\eta(t)}) &=  (\widehat{x}_t-\widehat{x}_{\eta(t)} ) \int_0^1 \partial_x g(t, \chi \widehat{x}_t + (1- \chi) \widehat{x}_{\eta(t)}, \chi  \widehat{v}_t +(1-\chi)  \widehat{v}_{\eta(t)})  \d \chi \\ & \qquad +
(\widehat{v}_t-\widehat{v}_{\eta(t)} ) \int_0^1 \partial_v g(t,\chi \widehat{x}_t + (1- \chi) \widehat{x}_{\eta(t)}, \chi  \widehat{v}_t +(1-\chi)  \widehat{v}_{\eta(t)})  \d \chi
\end{align*}
Using \eqref{est-g} and
$$ \frac{1}{\chi v_1 +(1-\chi)v_2} \leq \frac{1}{v_1}+ \frac{1}{v_2}, \qquad v_1,v_2 >0,$$
it follows that 
\begin{align} \label{est_err-2}
e_n^{(2)}  & \leq
 c  (1+M^q) \int_{t_n}^{t_{n+1}} \frac{1}{\sqrt{t}} E   |\widetilde{W}_t | ( |\widehat{v}_t-\widehat{v}_{\eta(t)}|  +   | \widehat{x}_t-\widehat{x}_{\eta(t)} |)  \Theta_t \d t
 \end{align}
with
 \begin{align*} 
 \Theta_t =  \frac{|I_t^B|}{\sqrt{t}} \left( \sqrt{ \widehat{v}_{\eta(t)}}  + \frac{1}{ \sqrt{ \widehat{v}_{t}}} \frac{\sqrt{ \widehat{v}_{\eta(t)}}}{\sqrt{ \widehat{v}_{t}}}  +  \frac{1}{ \sqrt{ \widehat{v}_{\eta(t)}}} \right), \qquad t \in [0,T]  
\end{align*}
Lemma \ref{inverse_moments} (1), (2), (5), (6) imply now 
that
$$ \sup_{t \in [0,T]}  \left( E \left| \frac{I_t^B}{\sqrt{t}} \right|^p \right)^{1/p} +  \sup_{t \in [0,T]} \left( E\left(   \sqrt{ \widehat{v}_{\eta(t)}}  + \frac{1}{ \sqrt{ \widehat{v}_{t}}} \frac{\sqrt{ \widehat{v}_{\eta(t)}}}{\sqrt{ \widehat{v}_{t}}}  +  \frac{1}{ \sqrt{ \widehat{v}_{\eta(t)}}} \right)^p \right)^{1/p} \leq c $$
for $2 \leq p < \frac{4 \kappa \lambda}{\theta^2}-2$.
Hence the Cauchy-Schwarz inequality yields
\begin{align} {E} \Theta_t^{1+\delta} \leq  c \label{theta_1} \end{align}
for $\delta \in \left(0, \frac{2\kappa \lambda}{\theta^2}-2 \right)$. Note that {\it (F)} ensures that the interval  for $\delta$ is non-empty.
Lemma \ref{smooth_inc} implies
\begin{align} \label{smooth_th}  \frac{1}{\sqrt{t}}  \left( E     |\widetilde{W}_t  |^q ( |\widehat{v}_t-\widehat{v}_{\eta(t)}|  +   | \widehat{x}_t-\widehat{x}_{\eta(t)} |)^q \right)^{1/q} \leq c  \frac{1}{\sqrt{t}} \Delta_t, \qquad t \in [0,T], \end{align}
for any $q \geq 1$. Hence  \eqref{est_err-2}, \eqref{theta_1}, \eqref{smooth_th}  and an application of H\"older's inequality give
\begin{align*} 
e_n^{(2)} \leq  c  (1+M^q)\int_{t_n}^{t_{n+1}} \frac{1}{\sqrt{t}} (t -\eta(t)) \d t
\end{align*}
}

Using \eqref{en_1_final} we  now obtain
$$ |e_n| \leq c (1+M^q) \Delta^2 + c  (1+M^q)\int_{t_n}^{t_{n+1}} \frac{1}{\sqrt{t}} (t -\eta(t)) \d t$$
Since $ [0,T] \ni t \mapsto \frac{1}{\sqrt{t}} \in (0, \infty)$ is integrable, it follows
\begin{align} |E(h({x}_N,v_N)-E(h(x_T,v_T))| \leq \sum_{n=1}^{N}|e_n| \leq c (1+M^q) \cdot \Delta \label{order_delta}, \end{align}
where $h(x,v)=f(x) \phi_M(v)$ with $f$ satisfying $(S)$ and the localizing function $\phi_M$.

\medskip

Now write 
$$ f(x)=h(x,v) + f(x)(1-\phi_M(v))$$
By construction we have
$$ f(x)(1- \phi_M(v)) = 0 \qquad \textrm{if} \qquad v \leq  M $$
and
$$ |f(x)(1- \phi_M(v))| \leq  \|f \|_{\infty} \left( :=\sup_{x \in \mathbb{R}} |f(x)| \right)    \qquad \textrm{if} \qquad v > M$$
The Markov inequality, Lemma \ref{CIR_moments} (2) and Lemma \ref{inverse_moments} (3) imply the existence of a constant $c_{tail}>0$ such that
 \begin{align*}  P( V_T \geq   M ) + P( v_N \geq  M ) \leq  c_{tail} \exp \left(-\frac{2 \kappa}{\theta^2} M \right)  \end{align*}
Hence we obtain
 $$ |E f(X_T)(1- \phi_M(V_T))| + |E f({x}_N)(1- \phi_M(v_N))| \leq c_{tail} \|f \|_{\infty} \exp \left(-\frac{2 \kappa}{\theta^2} M \right)  $$
Choosing $$M= - \frac{\theta^2}{2 \kappa} \log(\Delta) $$
and using \eqref{order_delta}
we end up with
$$   |E f(X_T) - Ef(x_N)| \leq   c \left(1 +  \Big( \frac{\theta^2}{2 \kappa}\Big)^q |\log(\Delta)|^q  \right) \cdot \Delta  + c_{tail}  \|f \|_{\infty} \cdot \Delta $$
which finishes the proof of Theorem \ref{theorem-main}.

\subsection{Proof of Theorem \ref{theorem-weak}}

Note that we only have  to show 
\begin{align} \label{to_show} E|x_N - X_T| \rightarrow 0 , \qquad \Delta \rightarrow 0,  \end{align}
since $L^1$-convergence implies convergence in probability, $X_T$ has a Lebesgue density, see e.g. \cite{dens}, and $f$ is continuous up to a finite number of points.
Assumption {\it (Int)} provides then the uniform integrability required to deduce
$$ Ef(x_N) \rightarrow E f(X_T), \qquad \Delta \rightarrow 0$$

To establish \eqref{to_show} write
\begin{align*}
x_N-X_T =  & - \frac{1}{2} \sum_{k=0}^{N-1} (v_k-V_{t_k}) (t_{k+1}-t_{k}) +   \sum_{k=0}^{N-1} (\sqrt{v_k}-\sqrt{V_{t_k}}) \Delta_k Z  
\\ &  +
 \frac{1}{2} \sum_{k=0}^{N-1} \int_{t_{k}}^{t_{k+1}} (V_t -V_{t_k}) \d t -  \sum_{k=0}^{N-1} \int_{t_{k}}^{t_{k+1}} (\sqrt{V_t} -\sqrt{V_{t_k}})
\d Z_t
\end{align*}
with the Brownian motion $Z = \rho W + \sqrt{1- \rho^2} B$. The It\=o isometry, the Minkowski and Lyapunov inequalities and $ |\sqrt{x} - \sqrt{y}|  \leq  \sqrt{|x-y|}$ for $x,y \geq 0 $ now yield
\begin{align*}
E |x_N-X_T| \leq  &  \sum_{k=0}^{N-1} E |v_k-V_{t_k}| (t_{k+1}-t_{k}) +  \sqrt{ \sum_{k=0}^{N-1}  E|v_k-V_{t_k}|  (t_{k+1}-t_{k}) }
\\ &  +
 \sum_{k=0}^{N-1} \int_{t_{k}}^{t_{k+1}} E |V_t -V_{t_k}| \d t +  \sqrt{ \sum_{k=0}^{N-1} \int_{t_{k}}^{t_{k+1}} E |V_t -V_{t_k}| \d t  }
\end{align*}
Lemma  \ref{CIR_moments} (3)  implies
\begin{align} \label{heston_strong}
E |x_N-X_T| \leq   \sum_{k=0}^{N-1} E |v_k-V_{t_k}| (t_{k+1}-t_{k}) +  \sqrt{ \sum_{k=0}^{N-1}  E|v_k-V_{t_k}|  (t_{k+1}-t_{k}) } + c \Delta^{1/4} \end{align}
for some constant $c>0$ independent of $\Delta$.
Using  the drift-implicit square-root Euler approximation  $a_k$ given by \eqref{sqrt_euler} and 
$ v_k \geq a_k$, see  \eqref{dom_sqrt},  we have
$$ E |v_k-V_{t_k}|  \leq   E|v_k -a_k|  + E |V_{t_k} - a_k| =  E(v_k-a_k) + E |V_{t_k} - a_k|, \qquad k=0,1, \ldots, N, $$
and thus \begin{align} \label{dimp_strong}
 E |v_k-V_{t_k}|   \leq  E(v_k-V_{t_k}) + 2 E |V_{t_k} - a_k|, \qquad k=0,1, \ldots, N  \end{align}
It remains to analyse the first summand on the right hand side of \eqref{dimp_strong}. Here we have
$$  E v_{k+1} =    E  v_{k} + \kappa (\lambda - E v_{k+1}) (t_{k+1}-t_k), \qquad k=0,1, \ldots, N-1,$$
which is the drift-implicit Euler approximation of 
$$ {E} V_t = v_0 +  \int_0^t \kappa(\lambda - E V_s) \,ds, \qquad t \in  [0,T],$$ and hence it follows
$$ \sup_{k=0, \ldots, N}  |{E} (v_k - V_{t_k})| \leq c \cdot \Delta$$
for some constant $c>0$ independent of $\Delta$.
This estimate, equation \eqref{dimp_strong} and Lemma 3.6 now give
$$   \sup_{k=0, \ldots, N}  E |v_k-V_{t_k}| \rightarrow 0, \qquad \Delta \rightarrow 0 $$
which finally together with \eqref{heston_strong} yields \eqref{to_show}.

\subsection{Proof of Proposition \ref{prop-int}}

Since 
$$ {v}_{k+1} - {v}_k - \kappa(\lambda-{v}_{k+1}) (t_{k+1}-t_k)  - \frac{\theta^2}{4} \big((\Delta_k W)^2 - (t_{k+1}-t_k)   \big) =  \theta \sqrt{{v}_k} \Delta_k W $$
for $k=0,1, \ldots$, we
have 
$$  \rho \sum_{k=0}^{N-1} \sqrt{{v}_k} \Delta_k W \leq \frac{|\rho|}{\theta} \left( v_0 + \kappa \lambda T +  \frac{\theta^2}{4} \sum_{k=0}^{N-1} (\Delta_k W)^2 \right)$$
if $\rho < 0$.
Thus we obtain for $$ x_N  =   - \frac{1}{2} \sum_{k=0}^{N-1} v_k (t_{k+1}-t_{k})  +  \rho \sum_{k=0}^{N-1} \sqrt{{v}_k} \Delta_k W + \sqrt{1 -\rho^2}   \sum_{k=0}^{N-1} \sqrt{v_k} \Delta_k B   $$ the upper bound
\begin{align*}
x_N =  & - \frac{1}{2} \sum_{k=0}^{N-1} v_k (t_{k+1}-t_{k})  +  \frac{|\rho|}{\theta} \left( v_0 + \kappa \lambda T \right) +  \frac{|\rho|\theta}{4} \sum_{k=0}^{N-1} (\Delta_k W)^2 + \sqrt{1 -\rho^2}   \sum_{k=0}^{N-1} \sqrt{v_k} \Delta_k B  
\end{align*}
and hence
$$ \exp( p \, x_N) \leq c \exp \left(   - \frac{p}{2}  \sum_{k=0}^{N-1} v_k (t_{k+1}-t_{k}) + p  \sqrt{1 -\rho^2}   \sum_{k=0}^{N-1} \sqrt{v_k} \Delta_k B + \frac{ p |\rho| \theta}{4}  \sum_{k=0}^{N-1} (\Delta_k W)^2 \right) $$ 
for some constant $c>0$ depending only on the parameters of the Heston model and $p,T$.
Since $v_k$, $k=0,1, \ldots,$ and $B$ are independent we have
$$ \sum_{k=0}^{N-1} \sqrt{v_k} \Delta_k B\,  \stackrel{\mathcal{L}}{=} \,  B_1 \sqrt{\sum_{k=0}^{N-1} v_k (t_{k+1}-t_k)} $$
and therefore
\begin{align*}
E  \exp( p \,  x_N)  &= E (E( \exp( p \, x_N) | W))  \\ &\leq c E   \exp \left(   \left(  \frac{p^2(1-\rho^2)}{2} -\frac{p}{2}  \right) \sum_{k=0}^{N-1} v_k (t_{k+1}-t_{k})   + \sum_{k=0}^{N-1} \frac{ p|\rho|\theta}{4} (\Delta_k W)^2   \right) 
\end{align*}
Note that
$$ \frac{ p^2(1-\rho^2) }{2}  - \frac{p}{2}  \leq 0 $$
iff
$$ p(1-\rho^2) \leq 1 $$ 
For $p=1+ \rho^2$  and $\rho <0$, this is satisfied and it follows
$$E  \exp( p \, x_N) \leq c  E  \exp \left(  \sum_{k=0}^{N-1} \frac{p |\rho| \theta}{4} (\Delta_k W)^2 \right) $$
The moment generating function of $W_1^2$ is given by
$$ E \exp (t W_1^2)= \exp \left( -\frac{1}{2} \ln(1-2t) \right), \qquad t < \frac{1}{2},$$  and we obtain
$$  E  \exp \left(  \sum_{k=0}^{N-1} \frac{p |\rho| \theta}{4} (\Delta_k W)^2 \right)  =   \exp \left( -\sum_{k=0}^{N-1}\frac{1}{2} \ln\left(1-  \frac{p |\rho|  \theta}{2} (t_{k+1}-t_k) \right) \right)   $$
for $\Delta < 2 / p |\rho| \theta $.  If $\Delta < 1/  p |\rho| \theta $ we have $1-  \frac{p |\rho|  \theta}{2} (t_{k+1}-t_k) \geq \frac{1}{2}$ for all $k=0, 1, \ldots, N-1$ and hence
$$ \ln\left(1-  \frac{p |\rho|  \theta}{2} (t_{k+1}-t_k) \right) \geq -p |\rho|  \theta (t_{k+1}-t_k), \qquad k=0,1, \ldots,N-1$$
Thus
it follows that
$$  E  \exp \left(  \sum_{k=0}^{N-1} \frac{p |\rho| \theta}{4} (\Delta_k W)^2 \right) \leq   \exp \left( T \frac{p |\rho| \theta}{2} \right), $$
which concludes the proof.

\bigskip
\bigskip

{\bf Acknowledgements.} \,\, The authors would like to thank Camelia A. Pop for  very helpful comments on the results from \cite{FP} and an unknown mathematician for pointing out a mistake in the proof of Theorem 1.1.

\bibliographystyle{plain}
\bibliography{refs}

\end{document}